\documentclass[11pt]{article}

\usepackage{hyperref}
\usepackage{tikz}
\usepackage{authblk}
\usepackage{enumerate}
\usepackage{mathrsfs}
\usepackage{amsmath}
\usepackage{amssymb}
\usepackage{wrapfig}
\usepackage[thmmarks]{ntheorem}
\usepackage{color}

\newlength{\bibitemsep}
\newlength{\bibparskip}
\let\oldthebibliography\thebibliography
\renewcommand\thebibliography[1]{%
  \oldthebibliography{#1}%
  \setlength{\parskip}{\bibitemsep}%
  \setlength{\itemsep}{\bibparskip}%
}

\renewcommand{\paragraph}{\roman{paragraph}}

\usepackage{hyperref}
\usepackage{tikz}
\usetikzlibrary{arrows,shapes,positioning}
\usetikzlibrary{decorations.markings}
\tikzstyle arrowstyle=[scale=1]
\tikzstyle directed=[postaction={decorate,decoration={markings, mark=at position .65 with {\arrow[arrowstyle]{stealth}}}}]
\tikzstyle reverse directed=[postaction={decorate,decoration={markings, mark=at position .65 with {\arrowreversed[arrowstyle]{stealth};}}}]

\newtheorem{theorem}{Theorem}[section]

\newtheorem{conjecture}[theorem]{Conjecture}

\newtheorem{lemma}[theorem]{Lemma}
\newtheorem{proposition}[theorem]{Proposition}
\newtheorem{claim}{Claim}
\newtheorem{subclaim}{Subclaim}[claim]

\newenvironment{proof}{\noindent {\bf
Proof.}}{\rule{3mm}{3mm}\par\medskip}

\newcommand{\phiv}{\varphi}
\newcommand{\phibar}{\bar{\varphi}}

\begin{document}
\title{Decomposition of class II graphs into two class I graphs\thanks{The work of the second author has been supported by the National Science Foundation through Grant DMS-2246292. The work of the third author has been supported by the Simons Foundation through the Grant No. 839830. The work of the forth author has been supported by the National Science Foundation through Grant CCF-2008422.}}

\author{
\small Yan Cao \thanks{School of Mathematical Sciences, Dalian University of Technology, Dalian, Liaoning, China, 116024},
Guangming Jing \thanks{Department of Mathematics, West Virginia University, Morgantown, WV, USA, 26505}, 
Rong Luo \thanks{Department of Mathematics, West Virginia University, Morgantown, WV, USA, 26505},
Vahan Mkrtchyan \thanks{Corresponding Author, Computer Science Department, Boston College, Chestnut Hill, MA, USA, 02467},
 Cun-Quan Zhang \thanks{Department of Mathematics, West Virginia University, Morgantown, WV, USA, 26505},
 Yue Zhao \thanks{Department of Mathematics, University of Central Florida, Orlando, FL, USA, 32816}
}
 \date{}
\maketitle
\vspace{-1cm}
\begin{abstract} Mkrtchyan and Steffen [J. Graph Theory, 70 (4), 473--482, 2012] showd that every class II simple graph can be decomposed into a maximum $\Delta$-edge-colorable subgraph and a matching. They further conjectured that every graph $G$ with chromatic index $\Delta(G)+k$ ($k\geq 1$) can be decomposed into a maximum $\Delta(G)$-edge-colorable subgraph (not necessarily class I) and a $k$-edge-colorable subgraph. In this paper, we first generalize their result to multigraphs and show that every multigraph $G$ with multiplicity $\mu$ can be decomposed into a maximum $\Delta(G)$-edge-colorable subgraph and a subgraph with maximum degree at most $\mu$. Then we prove that every graph $G$ with chromatic index $\Delta(G)+k$ can be decomposed into two class I subgraphs $H_1$ and $H_2$ such that $\Delta(H_1) = \Delta(G)$ and $\Delta(H_2) = k$, which is a variation of  their conjecture.
\end{abstract}

%\vspace{-0.2cm}

%\vspace{-0.4cm}

\section{Introduction}

%\vspace{-0.2cm}
Graphs considered in this paper are finite, undirected and  may contain multiple edges, but no loops. Let $V(G)$ and $E(G)$ denote the sets of vertices and edges of a graph $G$, respectively. For a vertex $x \in V$, let $N(x) =\{v~|~xv \in E(G)\}$ and $E(x)=\{e~|~ \mbox{$e$~is~incident~with~}x\}$. The degree of $x\in V(G)$ is denoted by $d(x)=|E(x)|$. The minimum and maximum degrees of vertices in $G$ are denoted by $\delta(G)$ and $\Delta(G)$, respectively.  A graph $G$ is regular if $\delta(G)=\Delta(G)$.  For two vertices $x$ and $y$, let $E(x,y)$ denote the set of edges  between $x$ and $y$ and let $\mu_G(x,y) = |E(x,y)|$. $\mu(G) = \max\{\mu_G(x,y)| ~x, y \in V(G)\}$ is called the  {\em multiplicity} of $G$.
 If there is no confusion from the context, we denote $\delta(G)$, $\Delta(G)$,  $\mu(G)$, $\mu_G(x,y)$ by $\delta$, $\Delta$,  $\mu$, and $\mu(x,y)$ respectively. 

An {\em edge coloring} of a graph is a function assigning values
(colors) to the edges  of the graph in such a way that any two
adjacent edges receive different colors. A graph is {\em $k$-edge-colorable} if there is an edge coloring of the graph with
colors from $C= \{1,\dots,k\}$.  The smallest integer $k$ such that $G$ is $k$-edge-colorable is called the {\em chromatic index} of $G$, and is denoted by $\chi'(G)$. Clearly $\chi'(G) \geq \Delta(G)$.

The classical theorems of Shannon and Vizing present upper bounds on $\chi'(G)$ as follows.
\begin{theorem}[Shannon \cite{Shannon:1949}]\label{shannon}
For any graph $G$, $\Delta(G)\leq \chi'(G) \leq \left \lfloor \frac{3\Delta(G)}{2} \right \rfloor$.
\end{theorem}
%\vspace{-0.4cm}

\begin{theorem}[Vizing \cite{vizing:1964}]
For any graph $G$, $\Delta(G)\leq \chi'(G) \leq \Delta(G)+\mu(G)$.
\end{theorem}

%\vspace{-0.2cm}
Let $G$ be a graph with maximum degree $\Delta$. 
It is said to be {\em class I} if $\chi'(G)=\Delta$, otherwise it is {\em class II}. There are many hard problems related to edge coloring of graphs (see \cite{stiebitz:2012}).   A subgraph $H$ of $G$   is  {\em maximum $\Delta$-colorable}  if it is $\Delta$-edge-colorable and  contains as many edges as possible. The maximum $\Delta$-edge-colorable subgraphs have been extensively studied (see  \cite{haas:1996, Rizzi:2009, samvel:2010, MkSteffen:2012}). A {\em decomposition} of a graph $G$ is a set $D = \{H_1,...\ , H_k\}$ of pairwise edge-disjoint subgraphs of $G$ that cover the set of edges of $G$.  Since the decision problem of classifying graphs as class I or class II is NP-complete, and every class II graph contains a subgraph of class I, it is natural to consider how one could decompose a class II graph into subgraphs with certain properties, for example, into subgraphs of class I with degree conditions. 
 
 Mkrtchyan and Steffen \cite{MkSteffen:2012} proved the following result.
 
%\vspace{-0.2cm}
 
 \begin{theorem}[Mkrtchyan and Steffen \cite{MkSteffen:2012}]
 \label{max-1}
 Let $G$ be a simple graph with  $\Delta(G)=\Delta$ and  $\chi'(G) = \Delta + 1$.  Then  $G$ can be decomposed into a  maximum $\Delta$-edge-colorable subgraph $H_1$ and a subgraph  $H_2$ with $\chi'(H_2) = \Delta(H_2) =1$.
 \end{theorem}
 
%\vspace{-0.2cm}
 They conjectured that Theorem~\ref{max-1} can be generalized to graphs with multiple edges.
 % \vspace{-0.2cm}
 \begin{conjecture}[Mkrtchyan and Steffen \cite{MkSteffen:2012}]
 \label{Conj1}
 Let $G$ be a graph with  $\Delta(G)=\Delta$ and  $\chi'(G) = \Delta + k$ where $k \geq 1$. Then  $G$ can be decomposed into a maximum $\Delta$-edge-colorable subgraph  $H_1$ and  a subgraph $H_2$ such that  $\chi'(H_2) = k$.
 \end{conjecture} 
%  \vspace{-0.2cm}
    
   In this paper, we first generalize Theorem~\ref{max-1} as follows.
 %  \vspace{-0.2cm}
   \begin{theorem}
\label{decom}
	 Let $G$ be a class II graph with multiplicity  $\mu$ and maximum degree $\Delta$. Then $G$ can be decomposed into two subgraphs $H_1$ and $H_2$ such that  $H_1$ is maximum $\Delta$-edge-colorable and $\Delta(H_2)\leq \mu$. Moreover, if $\mu\leq 2$ and $\chi'(G)=\Delta+\mu$, then $\chi'(H_2)=\Delta(H_2)=\mu$.
\end{theorem} 

%\vspace{-0.2cm}
In Theorem~\ref{decom}, if $G$ is simple and class II, then $\mu= 1$ and $\Delta(H_2) = 1 = \mu$. Thus it implies Theorem~\ref{max-1}.

Note that a graph has chromatic index one if and only if its maximum degree is one, and for simple graphs, it is easy to see that a maximum $\Delta$-edge-colorable subgraph  is class I by Vizing's adjacency lemma. However,  Vizing's adjacency lemma only works for simple graphs  and consequently, as shown in \cite{MkSteffen:2012}, maximum $\Delta$-edge-colorable subgraphs  could be  class II  for multigraphs. Thus in Conjecture~\ref{Conj1}, $H_1$  or $H_2$ could  be Class II.

Our second main result decomposes an edge colored graph into two Class I subgraphs, which is a variation of Conjecture~\ref{Conj1}.

%\vspace{-0.2cm}
\begin{theorem}
\label{main theorem}
Let $G$ be a graph with  $\Delta(G)=\Delta$ and  $\chi'(G)=\Delta+k$ where $k\geq 1$. Then $G$ can be  decomposed into two Class I subgraphs $H_1$ and $H_2$  such that $\chi'(H_1)=\Delta(H_1)=\Delta$ and $\chi'(H_2)=\Delta(H_2)=k$.
\end{theorem}

%\vspace{-0.2cm}
The proofs of Theorems~\ref{decom} and \ref{main theorem} will be presented in Section 3.

%\vspace{-0.4cm}
\section{Preliminaries and Lemmas}

%\vspace{-0.2cm}

In this section, we introduce additional notations and lemmas needed in the proofs of the theorems.
Let $G$ be a graph.  A $k$-{\sl vertex},
$k^+$-{\sl vertex}, or $k^-$-{\sl vertex} is a vertex of degree $k$, at least $k$, or at most $k$, respectively. We denote the set of all $k$-{\sl vertices}, $k^+$-{\sl vertices}, or $k^-$-{\sl vertices} in $V(G)$ by $V_{k}(G)$, $V_{\ge k}(G)$, or $V_{\le k}(G)$, respectively. 
For integers $r,s,t > 0$, let $T = T(r, s, t)$ be the graph consisting of three
vertices $x, y, z$ such that $\mu(x,y) = r$, $\mu(y,z) = s$, and $\mu(x,z) = t$. The graph $S_d = T(\lfloor \frac{d}{2} \rfloor,\lfloor \frac{d}{2} \rfloor,\lfloor \frac{d+1}{2} \rfloor)$ is called a {\it Shannon graph} of degree $d$. Vizing~\cite{Vizing} proved the following structural result for graphs achieving the upper bound in Theorem~\ref{shannon}.

%\vspace{-0.2cm}
\begin{theorem}[Vizing~\cite{Vizing}]\label{fat triangle}
Let $G$ be a graph with  $\Delta(G)=\Delta \geq 4$. If $\chi'(G)=\lfloor \frac{3}{2}\Delta \rfloor$, then $G$ contains a Shannon graph $S_\Delta$ as a subgraph.
\end{theorem}

%\vspace{-0.2cm}
We denote  $\mathcal{C}^k(G)$ to be the set of all $k$-edge colorings of a graph $G$. Let $\varphi\in \mathcal{C}^k(G)$ and $C= \{1,2,\dots, k\}$ be the  color set.  For a vertex $v\in V$, denote by $\varphi(v)=\{\varphi(e) | e\in E(v)\}$  the set of colors present at $v$ and
$\phibar(v) = C\setminus \varphi(v)$ the set of colors not assigned to any edge incident with $v$. A color $\gamma$ is said to be {\it missing} at $v$ if $\gamma\in\phibar(v)$. For a color $\alpha$, let
$E_{\alpha} =\{e\in E : \phiv(e) =\alpha\}$. For an edge set $E_0\subseteq E(G)$, let $\phiv(E_0)=\{c\in C~|~ \phiv(e)=c~for~some~e\in E_0\}$. Let us start with the following observation.

%\vspace{-0.2cm}
\begin{lemma}
\label{observation}
Let $G$ be a graph with $\Delta(G)=\Delta$ and $\chi'(G)=\Delta+k$ where $k \geq 1$. Let $x$ be a vertex in $G$. If there is a $(\Delta+k)$-edge coloring $\phiv$ of $G$ and a vertex $v\in V(G)$ such that $\phibar(x)\cap \phibar(v)=\emptyset$, then $G$ can be decomposed into two  Class I subgraphs $H_1$ and $H_2$ such that $\Delta(H_1)=d_G(x)$ and $\Delta(H_2)= \Delta + k -d_G(x)$.  In particular, if $d_G(x) = \Delta$, then $\Delta(H_1)= \Delta$ and $\Delta(H_2) = k$.
\end{lemma}
%\vspace{-0.4cm}
\begin{proof} 
 Assume without loss of generality that $\phiv(x)=\{1,2,\dots, d_G(x)\}$. Let 
 $$H_1=\bigcup_{i=1}^{d_G(x)} E_i ~~~\mbox{and}~~~H_2=\bigcup_{i=d_G(x)+1}^{\Delta+k} E_i.$$
  Clearly $\chi'(H_1)\le d_G(x)$ and $\chi'(H_2)\le \Delta + k - d_G(x)$. Note  $d_{H_1}(x)=d_G(x)$ and $\Delta + k \leq \chi'(H_1) + \chi'(H_2)$. Thus  $\chi'(H_1)=\Delta(H_1)=d_{H_1}(x)= d_G(x)$. Since $\phibar(x)\cap \phibar(v)=\emptyset$, we have $d_{H_2}(v)=\Delta + k -d_G(x)$. Therefore  $\chi'(H_2)=\Delta(H_2)=d_{H_2}(v)= \Delta + k - d_G(x)$, as desired.
\end{proof}
%\vspace{-0.2cm}

Let $G$ be a graph with an edge $e\in E(x,y)$, and let $\varphi$ be an edge coloring of $G-e$. A sequence
$F=(e_1, y_1, \ldots, e_p, y_p)$ consisting of vertices and distinct edges is called a
{\em multi-fan} at $x$ with respect to $e$ and $\varphi$ if $y_1=y$, $e_1=e$, and for each $2\leq i\leq p$, we have $e_i\in E(x,y_i)$  and 
$\phiv(e_i) \in \overline{\varphi}(y_j)$ for some $1\leq j\leq i-1$. A  {\em linear sequence}  from $y_1$ to $y_s$, denoted by $L=(e_1,y_1, e_2,y_2,\dots,e_s,y_s)$, is a sequence consisting of distinct vertices and distinct edges such that $e_i\in E(x,y_i)$ for $1\leq i\leq s$ and $\varphi(e_i)\in\overline\varphi(y_{i-1})$ for  each $2\leq i\leq p$. Clearly for any $y_i\in V(F)$, the multi-fan $F$ contains a linear sequence from $y$ to $y_i$, as we could just pick the shortest multi-fan contained in $F$ ending with $y_i$.

A {\em shifting} from $y_1$ to $y_s$ in a linear sequence $L=(e_1,y_1,\dots,e_s,y_s)$ is an operation that obtains a new edge coloring of $G-e_s$ from $\varphi$ by recoloring the edge $e_t$ with the color of $e_{t+1}$ under $\varphi$ sequentially for each $1\le t\leq s-1$ and uncoloring the edge $e_s$. 

An edge $e\in E(x,y)$ is called a {\em critical} edge of $G$ if $\chi'(G-e)<\chi'(G)$. Clearly if $e$ is critical, then $\chi'(G-e)=\chi'(G)-1$. We first have the  following result regarding multi-fans.

%\vspace{-0.2cm}
\begin{lemma}
\label{viz}
	Let $G$ be a graph with $\Delta(G) = \Delta$ and multiplicity $\mu$. 
	  Suppose that  $\chi'(G) \geq \Delta +1$  and $e \in E(x,y)$ is a critical edge of $G$. Let $F=(e_1, y_1, \ldots, e_p, y_p)$ be
	a maximal multi-fan at $x$ with respect to $e$ and $\varphi\in\mathcal{C}^{\chi'(G)-1}(G-e)$.  Then we have the following:
	
	\noindent
	(a) (Stiebitz et al.~\cite{stiebitz:2012})  $|V(F)|\geq 2$ and $\sum_{z\in V(F)}(d_G(z)+\mu_F(x,z)-(\chi'(G)-1))= 2$, where $\mu_F(x,z)$ is the number of edges between $x$ and $z$ in $F$.	
	
	\noindent
	 (b) If $d_G(y)\leq \chi'(G) -\mu$, then $F$ contains a vertex $z\neq y$ with $d_G(z)\geq \chi'(G)-\mu$. 	
\end{lemma}
%\vspace{-0.4cm}
\begin{proof} (a)  is Theorem 2.1 in book~\cite{stiebitz:2012} due to Stiebitz et al. 

(b)  	 Since $\mu\geq \mu_F(x,z)$,  by (a) we have $$2 = \sum_{z\in V(F)}(d_G(z)+\mu_F(x,z)-(\chi'(G)-1))\leq \sum_{z\in V(F)}(d_G(z)+\mu-(\chi'(G)-1)).$$
 Thus $F$ contains either a vertex $z$  with $d_G(z)+\mu-(\chi'(G)-1)> 1$, or two vertices $z_i$ with $i=1,2$ such that  $d_G(z_i)+\mu-(\chi'(G)-1)= 1$. In the former case we have $d_G(z)>(\chi'(G)-1)-\mu+1= \chi'(G)-\mu$,  so $z\neq y$ (because $d_G(y)\leq \chi'(G)-\mu$). In the latter case we have $d_G(z_i)\geq (\chi'(G)-1)-\mu+1= \chi'(G)-\mu$ and one of $z_1$ or $z_2$ is not $y$, as desired.
\end{proof}

%\vspace{-0.6cm}
\section{Decomposition of class II graphs}
%\vspace{-0.2cm}

In this section, we prove Theorems~\ref{decom} and ~\ref{main theorem}. We first would like to point out that  the proof would be much easier if there is no restriction on the maximum degrees of $H_1$ and $H_2$.

\begin{proposition}
\label{prop}
Every class II graph  can be decomposed into two class I graphs.
\end{proposition}
%\vspace{-0.4cm}
\begin{proof} Let $G$ be a graph with maximum degree $\Delta$ and $\chi'(G) = \Delta + k$ where $k\geq 1$.
Let $\varphi \in \mathcal{C}^{\Delta +k}(G)$ such that
$|E_1|$ is minimum, where $E_1$ is the set of edges in $G$ colored by the color $1$. Let $e\in E(x,y)$ and $e \in E_1$.

By the minimality of $|E_1|$, we have $\phibar(x)\cap \phibar(y) = \emptyset$. Otherwise one may recolor the edge $e$ with a color $i\in \phibar(x)\cap \phibar(y)$,  which contradicts the minimality of $|E_1|$.   By Lemma~\ref{observation},  $G$ can be decomposed into two class I subgraphs $H_1$ and $H_2$ such that $\Delta(H_1) = d_G(x) $ and $\Delta(H_2) = \Delta + k -d_G(x)$.
\end{proof}

We are ready to prove Theorem~\ref{decom}. 

\medskip
\noindent {\bf Theorem~\ref{decom}}  Let $G$ be a class II graph with multiplicity  $\mu$ and maximum degree $\Delta$. Then $G$ can be decomposed into two subgraphs $H_1$ and $H_2$ such that  $H_1$ is maximum $\Delta$-edge-colorable and $\Delta(H_2)\leq \mu$. Moreover, if $\mu\leq 2$ and $\chi'(G)=\Delta+\mu$, then $\chi'(H_2)=\Delta(H_2)=\mu$.

%\vspace{-0.2cm}
\begin{proof} We first prove the general case.  Take a maximum $\Delta$-edge-colorable subgraph $H_1$ of $G$ minimizing $$ t(H_2)= \sum_{d_{H_2}(v)>\mu}d_{H_2}(v),$$
where $H_2 = G-H_1$. It is sufficient to show $t(H_2) = 0$. Suppose to the contrary that $t(H_2) \geq 1$. Let  $y$ be a vertex  with $d_{H_2}(y)>\mu$ and  $e\in E(x,y)$ be an edge in $H_2$. 
	
	Since $H_1$ is maximum $\Delta$-edge-colorable, the edge $e$ is critical in $G'=H_1+e$ and $\chi'(G')=\Delta+1$. Let $\varphi\in\mathcal{C}^{\Delta}(H_1)$ and $F$ be a maximal multi-fan at $x$ with respect to $e$ and $\varphi$. Since $d_{H_2}(y)>\mu$ and $d_G(y)\leq \Delta$, we have 
	$$d_{G'}(y) = d_{H_1}(y) + 1= d_G(y) - d_{H_2}(y) +1\leq \Delta-(\mu +1) + 1<\Delta-\mu+1\leq \chi'(G')-\mu(G').$$
	 By Lemma~\ref{viz} (b) with $\chi'(G')=\Delta+1\geq \Delta(G')+1$, there is a vertex $z  \in V(F) \setminus \{y\}$  such that $d_{H_1}(z)\geq \chi'(G')-\mu(G')\geq\Delta+1-\mu$.  This implies $d_{H_2}(z) \leq \mu -1$. 
	
	By the definition of a multi-fan, $F$ has a linear sequence $L$ from $y$ to $z$ with the last edge $f\in E(x,z)$. By shifting colors along $L$, we see that $H_1'=G'-f=H_1+e-f$ is $\Delta$-edge-colorable and  has the same number of edges as $H_1$. Thus $H_1'$ is also maximum $\Delta$-edge-colorable.  Let $H_2' = G-H_1'$. Then $d_{H_2}(w) = d_{H_2'}(w)$ for each vertex $w \not = y,z$, $d_{H_2'} (y) = d_{H_2}(y)-1 \geq  \mu$ and $d_{H_2'}(z) = d_{H_2} (z) +1$. Since $d_{H_2}(z) \leq \mu -1$, we have $d_{H_2'} (z) \leq \mu$. Therefore 
	
	$$
	t(H_2')  =\left\{ \begin{array}{ll}
		t(H_2) -1 & \mbox{if $d_{H_2}(y) > \mu +1$, } \\
		t(H_2) -d_{H_2}(y) &  \mbox{ if $d_{H_2}(y) = \mu +1$.}
	\end{array}
	\right.
	$$
	
	In either case we have $t(H_2') < t(H_2)$, which contradicts the choice of $H_1$. This proves the general case.

	Now we prove the ``moreover" part.  Assume that $\mu\leq 2$ and $\chi'(G)=\Delta+\mu$. By the first part, let $H_1$ be a maximum $\Delta$-edge-colorable subgraph, such that $\Delta(H_2)=\Delta(G-H_1)\leq \mu$  where $H_2 = G-H_1$, and the number of odd cycles in $H_2$ is smallest. We will show  $\chi'(H_2) = \Delta(H_2) = \mu$. Note $1\leq \Delta(H_2) \leq \mu \leq 2$.
	
	If $\Delta(H_2) = 1$, then $\chi'(H_2) = 1$ and thus $\chi'(G) = \Delta + 1$, so $\mu = 1$. Therefore $\chi'(H_2) = \Delta(H_2) = \mu$.
	
	Now assume $\Delta(H_2) = 2$. Then $\mu =2$ and $H_2$ consists of vertex disjoint cycles and paths. To show $\chi'(H_2) = 2$, it is sufficient to show that $H_2$ contains no odd cycles. Suppose by contradiction that $H_2$ contains an odd cycle and $e \in E(x,y)$ is an edge in $E(H_2)$ lying in an odd cycle of $H_2$.

	Similar to the argument in the first part, $e$ is critical in $G'=H_1+e$ and $\chi'(G')=\Delta+1$. Let $\varphi\in\mathcal{C}^{\Delta}(H_1)$ and  $F$ be a maximal multi-fan at $x$ with respect to $e$ and $\varphi$. Since $d_{H_2}(y)=2$ and $d_G(y)\leq\Delta$, we have $d_{G'}(y)\leq \Delta-1= \Delta-\mu+1\leq \chi'(G')-\mu(G')$. Then by Lemma~\ref{viz} with $\chi'(G)=\Delta+1\geq\Delta(G')+1$, $F$ has a vertex $z\neq y$ such that $d_{H_1}(z)\geq \chi'(G)-\mu(G')\geq\Delta-1$. So $d_{H_2}(z)\leq 1$. This implies that $z$ is an endvertex of a path or a $0$-vertex in $H_2$.  Again $F$ has a linear sequence $L$ from $y$ to $z$ with last edge $f\in E(x,z)$, and by shifting colors along $L$, we see that $H_1'=G'-f=H_1+e-f$ is maximum $\Delta$-edge-colorable. We have that $\Delta(H_2') = 2$, where $H_2'=G-H_1'= H_2-e+f$, and the number of odd cycles in $H'_2$ is one less than in $H_2$. This contradicts the choice of $H_1$ and $H_2$ and thus proves the ``moreover" part. 
\end{proof}

For the proof of Theorem~\ref{main theorem}, we introduce additional notations. Let $\varphi\in \mathcal{C}^k(G)$.  For any two distinct colors $\alpha$ and $\beta$, let $G_{\varphi}(\alpha,\beta)$ be the subgraph of $G$ induced by $E_{\alpha} \cup E_{\beta}$. The components of $G_{\varphi}(\alpha,\beta)$ are called
{\it $(\alpha,\beta)$-chains}.  Clearly, each $(\alpha,\beta)$-chain is either a path or a cycle of edges alternately colored with $\alpha$ and $\beta$.
For each $(\alpha, \beta)$-chain $P$, let $\varphi/P$ denote the $k$-edge-coloring of $G$ obtained from $\varphi$ by exchanging colors $\alpha$ and $\beta$ on $P$.

For any $v\in V(G_{\varphi}(\alpha,\beta))$, let $P_v(\alpha, \beta, \varphi)$ denote
the unique $(\alpha, \beta)$-chain containing $v$. Note that, for any two vertices $u, \, v\in V(G_{\varphi}(\alpha,\beta))$, either $P_u(\alpha, \beta, \varphi)=P_v(\alpha, \beta, \varphi)$
or $P_u(\alpha, \beta, \varphi)$ is vertex-disjoint from $P_v(\alpha, \beta, \varphi)$. This fact will be used very often without mentioning. We are ready to prove Theorem~\ref{main theorem}.

\medskip
\noindent {\bf Theorem~\ref{main theorem}}  Let $G$ be a graph with  $\Delta(G)=\Delta$ and  $\chi'(G)=\Delta+k$ where $k\geq 1$. Then $G$ can be  decomposed into two Class I subgraphs $H_1$ and $H_2$  such that $\chi'(H_1)=\Delta(H_1)=\Delta$ and $\chi'(H_2)=\Delta(H_2)=k$.

\begin{proof}
We prove the theorem by contradiction. Let $G$ be a counterexample minimizing $\Delta=\Delta(G)$. Clearly, $\Delta \geq 3$.

 Let $x \in V(G)$ be a $\Delta$-vertex.  We have the following claim.
 
\begin{claim}\label{claim1}
For any $y\in V_{\ge \Delta-1}(G)\setminus \{x\}$, we have $|E(x,y)|\ge d(y)-k+1$. Consequently $V_{\ge \Delta-1}(G) \subseteq N(x)\cup \{x\}$.
\end{claim}
\begin{proof}
Suppose to the contrary that $|E(x,y)|< d(y)-k+1$. Let $\phiv\in C^{\Delta+k}(G)$ be a coloring such that $|\phibar(x)\cap \phibar(y)|$ is minimum.

Define: 
$$C_1 = \phibar(x)\cap \phibar(y) = \overline{\phiv(x)\cup \phiv(y)},$$   $$C_2=\phibar(x)\setminus \phibar(y)= \phiv(y)\setminus\phiv(x),$$
$$C_3=  \phibar(y)\setminus \phibar(x)= \phiv(x)\setminus\phiv(y),$$   $$~~~\mbox{and} ~~C_4=\phiv(x)\cap \phiv(y).$$ 

It is easy to see that  $C_1, C_2,C_3,C_4$ are pairwise disjoint and   form a partition of the  color set $C=\{1,2,\dots,\Delta+k\}$.  By Lemma~\ref{observation}, we have $C_1 = \phibar(x)\cap \phibar(y)\not=\emptyset$.

\begin{subclaim}\label{sc1}
$P_x(\alpha,\eta,\phiv)=P_y(\alpha,\eta,\phiv)$ for any two colors $\alpha\in C_1$ and $\eta\in C_4$.  
\end{subclaim}
%\vspace{-0.4cm}
\begin{proof}
Otherwise  let $\phiv'=\phiv/P_y(\alpha,\eta,\phiv)$. Then $\phibar'(x)\cap\phibar'(y)=(\phibar(x)\cap \phibar(y)) \setminus \{\alpha\}$,  a contradiction to the minimality of $|\phibar(x)\cap \phibar(y)|$.
\end{proof}

%\vspace{-0.2cm}
\begin{subclaim}
\label{sc2}
$\phiv(E(x)\setminus E(x,y)) \cap C_4 \not = \emptyset$. 
\end{subclaim}
%\vspace{-0.4cm}
\begin{proof}
Suppose to the contrary $\phiv(E(x)\setminus E(x,y)) \cap C_4  = \emptyset$. Since  $\phiv(E(x,y)) \subseteq C_4$, we have $C_4=\phiv(E(x,y))$. This implies that $C_2=\phiv(E(y)\setminus E(x,y))$ and $C_3=\phiv(E(x)\setminus E(x,y))$. Thus $|C_4|=|E(x,y)|$, $|C_2|=d(y)- |E(x,y)|$, and $|C_3|=d(x)-|E(x,y)|=\Delta-|E(x,y)|$. Therefore
\begin{align*}
	\Delta+k =|C| &=|C_1|+|C_2|+|C_3|+|C_4|\\
	 &=|C_1|+\Delta+d(y)-|E(x,y)|\\
	&> 1+\Delta+d(y)-(d(y)-k+1)\\
	&=\Delta+k.
\end{align*}

 This  contradiction proves the subclaim.
  \end{proof}
%\vspace{-0.2cm}

By Subclaim~\ref{sc2}, let $z_1 \in N(x)\setminus \{y\}$  and $e_1 \in E(x,z_1)$ such that $\phiv(e_1) = \eta \in C_4$.

%\vspace{-0.2cm}
\begin{subclaim}\label{sc3}
 $\phibar(z_1)\cap (C_1\cup C_4) =\emptyset$. Therefore $\phibar(z_1) \subseteq C_2\cup C_3$.
\end{subclaim}
%\vspace{-0.4cm}
\begin{proof}
Suppose to  the contrary that there exists a color $\alpha\in \phibar(z_1)\cap (C_1\cup C_4) $. 

If $\alpha \in C_1$, then  $P_x(\alpha,\eta,\phiv)=P_{z_1}(\alpha,\eta,\phiv)=xz_1$,  a contradiction to Subclaim~\ref{sc1}.

Assume $\alpha \in C_4$. Let $\beta \in C_1$. Then $P_{z_1}(\beta,\alpha,\phiv)$ does not contain $x,y$ by Subclaim~\ref{sc1}. Let $\phiv'=\phiv/P_{z_1}(\beta,\alpha,\phiv)$. Then $\beta\in \phibar'(z_1)$ and $\beta \in \phibar'(x)\cap \phibar'(y)$. Note $\phibar'(x)\cap \phibar'(y) =  \phibar(x)\cap \phibar(y)$. Thus  we are back to the previous case.
\end{proof}

%\vspace{-0.2cm}

By Subclaim~\ref{sc3}, $\phibar(z_1)\subseteq C_2\cup C_3$. Since $|\phibar(z_1)|=\Delta+k-d(z_1)\ge k$ and $|C_2|=|\phibar(x)|-|C_1|=k-|C_1|<k$, we have $\phibar(z_1)\cap C_3\not=\emptyset$. 

Let $\gamma_1$ be a color in $\phibar(z_1)\cap C_3$. Then there are $z_2\in N(x)$ and $e_2 \in E(x,z_2)$ such that $\phiv(xz_2)=\gamma_1$. Since $\gamma_1\in \phibar(z_1)\cap \phibar(y)$, we have $z_2\not=z_1$ and $z_2\not= y$.

%\vspace{-0.2cm}
\begin{subclaim}\label{sc5}
$\phibar(z_2)\subseteq C_2\cup C_3= (\phibar(x)\cup \phibar(y))\setminus (\phibar(x)\cap \phibar(y))$.
\end{subclaim}
%\vspace{-0.4cm}
\begin{proof}
Suppose to the contrary that there is a color $\alpha \in \phibar(z_2)$ with $\alpha \not \in C_2\cup C_3$.

If $\alpha \in C_1$, then  $P_{x}(\alpha,\gamma_1,\phiv)=P_{z_2}(\alpha,\gamma_1,\phiv)=xz_2$. Thus $P_{z_1}(\alpha,\gamma_1,\phiv)$ does not contain $x$. Note that both $\alpha$ and $\gamma_1$ are missing at $y$. Let $\phiv'=\phiv/P_{z_1}(\alpha,\gamma_1,\phiv)$. Then $\alpha\in \phibar'(z_1)$ and $\alpha \in \phibar'(x)\cap \phibar'(y)$,  a contradiction to Subclaim~\ref{sc3}.

Now assume $\alpha \in C_4$. Let $\beta$ be a color in $C_1$.  
Then $P_{z_2}(\beta,\alpha,\phiv)$ does not contain $x$ or $y$ by Subclaim~\ref{sc1}. Let $\phiv'=\phiv/P_{z_2}(\beta,\alpha,\phiv)$. Then $\beta \in \phibar'(z_2)$ and $\beta \in \phibar'(x)\cap\phibar'(y)=C_1$, where we have reached the case above with $\beta$ replacing $\alpha$.
\end{proof}
%\vspace{-0.2cm}

Note that $|\phibar(z_1)|+|\phibar(z_2)|=2(\Delta+k)-d(z_1)-d(z_2)\ge 2k$ and $|C_2\cup C_3|=|\phibar(x)|+|\phibar(y)|-2|C_1|\le k+(k+1)-2=2k-1$.   By Subclaims~\ref{sc3} and \ref{sc5}, $\phibar(z_1) \cup \phibar(z_2) \subseteq C_2\cup C_3$. Thus $\phibar(z_1)\cap \phibar(z_2)\not=\emptyset$. Let $\beta$ be a color in $\phibar(z_1)\cap \phibar(z_2)$. Then $\beta \in C_2\cup C_3$. Let $\alpha$ be a color in $C_1$.

If $\beta \in C_2$, then $\beta,\alpha\in \phibar(x)$ and at least one of $P_{z_1}(\beta,\alpha,\phiv),P_{z_2}(\beta,\alpha,\phiv)$ does not contain $y$.

 If $\beta\in C_3$, then $\beta,\alpha\in \phibar(y)$ and at least one of $P_{z_1}(\beta,\alpha,\phiv)$ and $P_{z_2}(\beta,\alpha,\phiv)$ does not contain $x$.
 
  In either case, at least one of $P_{z_1}(\beta,\alpha,\phiv),P_{z_2}(\beta,\alpha,\phiv)$ does not contain $x$ or $y$. 
  
  If $P_{z_1}(\beta,\alpha,\phiv)$ does not contain $x$ or $y$,  let $\phiv'=\phiv/P_{z_1}(\beta,\alpha,\phiv)$. Then $\alpha\in \phibar'(z_1)$ and $\alpha \in \phibar'(x)\cap \phibar'(y)$,  a contradiction to Subclaim~\ref{sc3}. 
  
  If  $P_{z_2}(\beta,\alpha,\phiv)$ does not contain $x$ or $y$, let $\phiv'=\phiv/P_{z_2}(\beta,\alpha,\phiv)$. Then $\alpha\in \phibar'(z_2)$ and $\alpha \in \phibar'(x)\cap \phibar'(y)$, a contradiction to Subclaim~\ref{sc5}. This completes the proof of Claim~\ref{claim1}.
\end{proof}
%\vspace{-0.2cm}

By Claim~\ref{claim1}, every $(\Delta-1)^+$-vertex is either $x$ or is adjacent to $x$.  The next claim shows that there are at most two $(\Delta-1)^+$-vertices in $G$.

%\vspace{-0.2cm}
\begin{claim}\label{claim2}
There is at most one  $(\Delta -1)^+$-vertex in $N(x)$.
\end{claim}
%\vspace{-0.4cm}
\begin{proof}
Suppose to the contrary that there exist two  $(\Delta-1)^+$-vertices $y,z \in N(x)$. By Claim~\ref{claim1}, 
 $$\Delta = d(x) \geq |E(x,y)| + |E(x,z)| \geq d(y)-k+1 + d(z)-k+1 \geq 2\Delta - 2k.$$
 Thus $k \geq \frac{\Delta}{2}$.  On the other hand,  we know $k\le \frac{\Delta}{2}$ by Theorem~\ref{shannon}.  This implies that $k=\frac{\Delta}{2}$, $\Delta$ is even, $d(y)=d(z)=\Delta-1$, and $N(x)=\{y,z\}$. Thus by Claim~\ref{claim1}, we have $V_{\ge \Delta-1}(G)=\{x,y,z\}$. However, Theorem~\ref{fat triangle} implies that $G$ must contain a Shannon graph of degree $\Delta$ as a subgraph which contains at least three $\Delta$-vertices,  a contradiction. 
\end{proof}

Let $y\in N(x)$ such that $d_G(y) = \max\{d_G(z) | z \in N(x)\}$. Define $e \in E(x,y)$. Then, each vertex  in $V(G)-\{x,y\}$ has degree at most $\Delta -2$.
 Let $\phiv \in \mathcal{C}^{\Delta+k}(G)$ and   $\alpha = \phiv(e)$.  Let $G'= G- E_{\alpha}$. Then $\Delta(G')=\Delta-1$ and $\chi'(G')=\Delta+k-1$. Thus by the minimality of $\Delta$,  $G'= G-E_\alpha$ can be decomposed into two class I subgraphs $H_1'$ and $H_2'$ such that $\Delta(H_1')=\Delta-1$ and $\Delta(H_2')=k$. Let $H_1=H_1'+E_\alpha$ and $H_2=H_2'$. Since $V_{\ge \Delta-1}(G)\subseteq \{x,y\}$, we have $\Delta(H_1)=\Delta(H_1')+1=\Delta$. Thus $H_1$ and $H_2$ form a desired decomposition. This completes the proof of Theorem~\ref{main theorem}.
\end{proof}

%\vspace{-0.6cm}
\section{Concluding remarks}
%\vspace{-0.3cm}

Theorem~\ref{main theorem} states that the  maximum degree of the two subgraphs are $\Delta$ and $k$ respectively.  
We believe that it can be generalized as follows.
%\vspace{-0.2cm}

\begin{conjecture}  Let $G$ be a graph with $\chi'(G) = \Delta + k$ and $p,q$ be  two positive integers such that  $p+q = \Delta + k$ with $p,q\leq \Delta$.
Then $G$ can be decomposed into two class I subgraphs $H_1$ and $H_2$ such that  $\Delta(H_1) = p$ and $\Delta(H_2) = q$.
\end{conjecture} 
%\vspace{-0.3cm}

Regarding maximum $\Delta$-edge-colorable subgraphs, Mkrtchyan and Steffen \cite{MkSteffen:2012} proved the following.

%\vspace{-0.2cm}
  \begin{theorem}[Mkrtchyan and Steffen \cite{MkSteffen:2012}] 
 \label{max-2}
 Let $G$ be a  bridgeless cubic graph and $M$ be a perfect matching of $G$. Then there is a maximum $3$-edge-colorable subgraph $H$ such that $M\cup E(H) = E(G)$.
  \end{theorem} 

We believe that  Theorem~\ref{max-2} can be generalized as follows:

\begin{conjecture} Let $G$ be a simple graph. Then for each maximum matching $M$ of $G$, there is a maximum $\Delta$-colorable subgraph $H$ of $G$ with $M\cup E(H)=E(G)$.
\end{conjecture}

\begin{conjecture} Let $G$ be an $r$-regular graph with $\chi'(G)=\Delta+1$. Then for each maximum matching $M$ of $G$, there is a maximum $\Delta$-colorable subgraph $H$ of $G$ with $M\cup E(H)=E(G)$.
\end{conjecture}

\setlength{\bibitemsep}{0.1\baselineskip plus .05\baselineskip minus .05\baselineskip}

\end{document}